\documentclass[11pt]{amsart}
\usepackage{amsmath,amssymb,amsthm, hyperref, amscd}
\usepackage{enumitem}
\usepackage[T1]{fontenc}
\usepackage{textcomp}
\usepackage[scale=0.95,ttdefault]{AnonymousPro}
\hypersetup{
   colorlinks=true,
   linkcolor=blue,
    urlcolor=cyan,
    }
\usepackage[
    backend=biber,
    style=alphabetic,
    sorting=nyt,
    useprefix=true
]{biblatex}
\usepackage{hyperref}

\newcommand{\fm}{\mathfrak{m}}
\newcommand{\fn}{\mathfrak{n}}

\newcommand{\fM}{\mathfrak{M}}

\newcommand{\fp}{\mathfrak{p}}
\newcommand{\fq}{\mathfrak{q}}

\newcommand{\Max}{\operatorname{Max}}
\newcommand{\Spec}{\operatorname{Spec}}

\newcommand{\colim}{\operatorname{colim}}

\newcommand{\Hom}{\operatorname{Hom}}
\newcommand{\cO}{\mathcal{O}}

\newcommand{\bC}{\mathbf{C}}

\newcommand{\bQ}{\mathbf{Q}}

\newcommand{\bZ}{\mathbf{Z}}

\newcommand{\Tag}[1]{\href{https://stacks.math.columbia.edu/tag/#1}{\texttt{#1}}}
\newcommand{\citestacks}[1]{\cite[Tag \Tag{#1}]{stacks}}
\newcommand{\citetwostacks}[2]{\cite[Tags \Tag{#1} and \Tag{#2}]{stacks}}

\newtheorem{Thm}{Theorem}%[section]

\newtheorem{Lem}[Thm]{Lemma}
\newtheorem{Cor}[Thm]{Corollary}
\newtheorem{Prop}[Thm]{Proposition}

\theoremstyle{definition}
\newtheorem{Def}[Thm]{Definition}

\newtheorem{Exam}[Thm]{Example}

\theoremstyle{remark}

\newtheorem{Rem}[Thm]{Remark}
\newtheorem{SteplocUJnotUJ}{Step}

\title[Universally Japanese]{The Japanese and universally Japanese properties for valuation rings and Pr\"ufer domains}

\author{Shiji Lyu}
\address{Department of Mathematics, Statistics, and Computer Science\\University of Illinois at Chicago\\Chicago, IL
60607-7045\\USA}
\email{\href{mailto:slyu@uic.edu}{slyu@uic.edu}}
\urladdr{\url{https://homepages.math.uic.edu/~slyu/}}

\begin{document}
\begin{abstract}
We discuss the Japanese and universally Japanese properties for valuation rings and Pr\"ufer domains.
These properties, regarding finiteness of integral closure,
have been studied extensively for Noetherian rings,
but very rarely, if ever,
for non-Noetherian rings.
Among other results,
we show that for valuation rings and Pr\"ufer domains,
the Japanese and universally Japanese properties are equivalent.
This result can be seen as a counterpart to Nagata's classical result for Noetherian rings.
This result also tells us many non-Noetherian rings,
including all absolutely integrally closed valuation rings and Pr\"ufer domains,
are universally Japanese.
\end{abstract}
\maketitle

Throughout, all rings are commutative with unit.
$\Max(R)$ denotes the set of maximal ideals of a ring $R$.
For $f\in R$,
$D(f)$ denotes the principal open subset of $\Spec(R)$ defined by $f$.
For $\fp\in\Spec(R)$, $\kappa(\fp)$ denotes the residue field of $\fp$.

A \emph{valuation ring} is a local \emph{domain} whose finitely generated ideals are principal.
A \emph{Pr\"ufer domain} is an integral domain whose localizations at prime ideals are valuation rings.
For basic properties we refer the reader to \cite[p. 25]{Coherent-Rings}.
A PID is an integral domain all ideals of which are principal.
A DVR is a local PID.

An integral domain is \emph{absolutely integrally closed} if it is normal and its fraction field is algebraically closed.
The \emph{absolute integral closure} of an integral domain is its integral closure in an algebraic closure of its fraction field.

For a ring $R$, we denote by $R(x_1,\ldots,x_n)$
the localization of the polynomial ring $R[x_1,\ldots,x_n]$
with respect to all primitive polynomials,
that is, all polynomials whose coefficients generate the unit ideal.
We denote by $R_{\operatorname{red}}$ the reduction of $R$.

\section{Introduction}

\begin{Def}    
An integral domain $S$ is \emph{N-1} if its integral closure in its fraction field is finite.
    An integral domain $S$ is \emph{N-2} (or \emph{Japanese}) if its integral closure in every finite extension of its fraction field is finite.

        A ring $R$ is \emph{universally Japanese} if every finite type $R$-algebra $A$ that is an integral domain is N-1.
        Equivalently, if every finite type $R$-algebra $A$ that is an integral domain is N-2.
\end{Def}

Most Noetherian rings appearing in algebraic geometry are universally Japanese, see \citestacks{0335}.
Since the property only concerns integral domains over the ring, it is possible to get non-Noetherian examples ``for free.''

\begin{Exam}
    An infinite direct product $A$ of fields is universally Japanese, as every prime ideal of $A$ is maximal (see for example \citetwostacks{092F}{092G}).
\end{Exam}

However, it is not super easy to find non-Noetherian universally Japanese integral domains.
\begin{Exam}\label{exam:infVarsNotUJ}
    Let $R=k[x_1,\ldots,x_n,\ldots]$ be the polynomial ring of infinitely many variables over a field $k$.
    Then $R$ is not universally Japanese,
    since the subring
$S=k[x_1^2,x_1^3,\ldots,x_n^2,x_n^3,\ldots]$
has normalization $R$, which is not finite over $S$,
and $S$ is isomorphic to a quotient of $R$.

On the other hand,
it is not hard to show that a \emph{finitely presented} $R$-algebra that is an integral domain is N-2,
as $R$ is a filtered colimit of finite type $k$-algebras with smooth transition maps,
see Lemma \ref{lem:limitN2}.
\end{Exam}

In this article, we discuss the N-2 and universally Japanese properties for valuation rings and Pr\"ufer domains,
thereby
finding classes of non-Noetherian universally Japanese integral domains.

Our main result is as follows.
This result can be seen as a Pr\"ufer domain counterpart to the classical result \cite[Th\'eor\`eme 7.7.2]{EGA4_2} of Nagata.
\begin{Thm}\label{thm:smallPrUJ}
    Let $D$ be a Pr\"ufer domain.
    Then the following are equivalent.
    \begin{enumerate}
       [label=$(\roman*)$]
        \item\label{smallPrUJ:N2} $D$ is N-2. 
        \item\label{smallPrUJ:uJ} $D$ is universally Japanese.
    \end{enumerate}
\end{Thm}

Specializing to valuation rings, we have the following result. 
Note that
the valuation-theoretic characterization of the N-2 property is essentially done in \cite[Chap. VI, \S 8, no.4-5]{Bourbaki-CA-5-7},
see Lemma \ref{lem:ValN2}.

\begin{Thm}\label{thm:ValRingUJ}
    Let $V$ be a valuation ring.
    Then the following are equivalent.
    \begin{enumerate}
       [label=$(\roman*)$]
        \item\label{VRUJ:N2} $V$ is N-2.
        \item\label{VRUJ:uN2} $V(x_1,\ldots,x_n)$ is N-2 for all $n\in\bZ_{\geq 0}$. 
        \item\label{VRUJ:sN2} $V(x_1,\ldots,x_n)$ is N-2 for some $n\in\bZ_{> 0}$. 
        \item\label{VRUJ:uJ} $V$ is universally Japanese.
    \end{enumerate}
\end{Thm}

Since every absolutely integrally closed domain is N-2 by definition, 
Theorem \ref{thm:smallPrUJ} implies that an absolutely integrally closed Pr\"ufer domain
is universally Japanese. 
However, the proof of 
Theorem \ref{thm:smallPrUJ} is in fact done by reduction to the absolutely integrally closed case.
We will show

\begin{Thm}\label{thm:PrufUJ}
    Let $D$ be a Pr\"ufer domain.
    Assume that for all maximal ideals $\fm$ of $D$, $D_\fm$ has divisible value group and $D/\fm$ is perfect.
    Then the following are equivalent.
    \begin{enumerate}
       [label=$(\roman*)$]
       \item\label{PrUJ:maxN2} $D_\fm$ is N-2 for all maximal ideals $\fm$ of $D$.
       \item\label{PrUJ:locN2} $D_\fp$ is N-2 for all prime ideals $\fp$ of $D$.
        \item\label{PrUJ:N2} $D$ is N-2. 
        \item\label{PrUJ:uJ} $D$ is universally Japanese.
    \end{enumerate}
\end{Thm}

As a consequence we see
\begin{Thm}\label{thm:ExampleUJ}
    The following rings are universally Japanese.
    \begin{enumerate}[label=$(\roman*)$]
        \item\label{ExUJ:absVal} An absolutely integrally closed valuation ring.
        \item\label{ExUJ:char0Val} A valuation ring of residue characteristic zero with divisible value group.
        \item\label{ExUJ:absPr} An absolutely integrally closed Pr\"ufer domain.
        \item\label{ExUJ:char0Pr} A Pr\"ufer domain containing the field of rational numbers $\bQ$ whose local rings 
        at maximal ideals 
    have divisible value groups.
        \item\label{ExUJ:absclosVal} The absolute integral closure of a valuation ring.
        \item\label{ExUJ:absclosPr} The absolute integral closure of a Pr\"ufer domain.
        \item\label{ExUJ:absclosDdkd} The absolute integral closure of a Dedekind domain.
        \item\label{ExUJ:absclosNoe1} The absolute integral closure of a Noetherian integral domain of dimension $1$.
    \end{enumerate}
\end{Thm}

\begin{Rem}\label{rem:UJnotlocal}
The equivalence of \ref{PrUJ:maxN2} and \ref{PrUJ:N2} in Theorem \ref{thm:PrufUJ}
does not hold for general Pr\"ufer domains.
     In fact, it does not even hold for general PIDs.
     See \cite[Proposition 2.5]{Heitmann-PID-locEx-notEx}.
     For more discussion, see \S\ref{sec:ex}.
\end{Rem}

\begin{Rem}
    The ring $\cO(\bC)$ of complex entire functions is not universally Japanese.
    Indeed, for a maximal ideal $\fm$ not of the form $(z-a)$ for some $a\in\bC$,
    the value group of $\cO(\bC)_\fm$ is a non-principal ultraproduct of $\bZ$.
    Therefore $\cO(\bC)_\fm$ is not N-2 by Lemma \ref{lem:ValN2},
    same for $\cO(\bC)$ by \citestacks{032G}.
\end{Rem}

We conclude by reminding the reader that there are, in fact, many absolutely integrally closed domains that are not universally Japanese.
\begin{Thm}\label{thm:absIntNotUJ}
    The absolute integral closure of a locally Nagata Noetherian integral domain containing $\bQ$ and of dimension $\geq 2$ is not universally Japanese.
    The absolute integral closure of a locally Nagata Noetherian integral domain containing $\bZ$ and of dimension $\geq 3$ is not universally Japanese.
\end{Thm}
For example, the absolute integral closure of $\bC[X,Y]$, $\bZ[X,Y]$, or $\bC[[X,Y]]$ is not universally Japanese.

As a side note, many such rings are also not coherent \cite{P-coherent-abs-int-clos}.
Coherence is a key input in our argument, see \S\ref{subsec:S2ify}.
\\

Our proof strategy is as follows.
In \S\ref{subsec:valN2},
after collecting information from \cite{Bourbaki-CA-5-7},
we characterize N-2 valuation rings.
Together with \cite{K-DVR-defectless},
we see $V(x_1,\ldots,x_n)$ is N-2 for every N-2 valuation ring $V$.
Using a naive variant of Serre's criterion for normality (\S\ref{subsec:Serre}) and of the dualizing module (\S\ref{subsec:S2ify}),
we can reduce the problem for $V[x_1,\ldots,x_n]$ to $V(x_1,\ldots,x_n)$ (Corollary \ref{cor:PruferSerre})
and prove Theorem \ref{thm:ValRingUJ}
relatively easily.
For a general Pr\"ufer domain $D$,
the same reduction process works,
but to control all $D_\fm(x_1,\ldots,x_n)$ simultaneously is not possible in general (Remark \ref{rem:UJnotlocal}).
In the situation of
Theorem \ref{thm:PrufUJ},
it is possible,
and Theorems \ref{thm:ExampleUJ} and \ref{thm:smallPrUJ} follow.

The proof of Theorem \ref{thm:absIntNotUJ} is on its own and not related to the materials on Pr\"ufer domains.

In \S\ref{sec:ex}, we examine the failure of the equivalence of \ref{PrUJ:maxN2} and \ref{PrUJ:N2} in Theorem \ref{thm:PrufUJ} in arbitrary characteristic,
and supplement that with another situation in which the equivalence does hold.
This tells us the assumptions in Theorem \ref{thm:PrufUJ} may be weakened but cannot be removed.
\\

\noindent\textsc{Acknowledgment}. We thank Linquan Ma for informing the author of this problem and for helpful discussions.
We thank the anonymous referee for detailed comments that help improve the exposition.
The author was supported by an AMS-Simons Travel Grant.

\section{Preparations}
\subsection{Remarks on ascent and descent of integral closure}
\begin{Def}
    A ring map $A\to B$ is \emph{normal} if it is flat and all its fibers are geometrically normal,
    that is, $B\otimes_A l$ is always normal where $l$ is a finite extension of $\kappa(\fp)=A_\fp/\fp A_\fp$ for some $\fp\in\Spec(A)$.
\end{Def}
Note that a normal ring map, or any flat ring map, between integral domains, is injective.

For a property \textbf{P} of ring maps, a ring map $A\to B$ is \emph{essentially \textbf{P}} if $B=S^{-1}C$ where $A\to C$ is \textbf{P}.
A ring map $A\to B$ is of finite presentation if $B$ is isomorphic to a polynomial $A$-algebra of finitely many variables modulo a finitely generated ideal.
\begin{Lem}
    \label{lem:basechangeNormal}
    Let $A\to B$ be a normal ring map essentially of finite presentation.
    Let $C$ be an $A$-algebra.
    Then $C\to C\otimes_A B$ is a normal ring map essentially of finite presentation.
\end{Lem}
\begin{proof}
    The Noetherian case is contained in \cite[Proposition 6.8.2]{EGA4_2},
    whereas the general case follows from the same proof.
\end{proof}

\begin{Lem}[{\cite[Proposition 11.3.13]{EGA4_3}}]\label{lem:ascendNormal}
    Let $A\to B$ be a normal ring map essentially of finite presentation (for example essentially smooth).
    Assume $A$ is normal.
    Then $B$ is normal.
\end{Lem}

\begin{Cor}\label{cor:NormalizationBaseChange}
Let $A\to B$ be a normal ring map essentially of finite presentation.
Let $C$ be an $A$-algebra,
and let $C'$ be the integral closure of $A$ in $C$.
If $C'$ is normal, then $B\otimes_A C'$ is the integral closure of $B$ in $B\otimes_A C$.
\end{Cor}
\begin{proof}
    As $A\to B$ is flat, $B\otimes_A C'$ is a subring of $B\otimes_A C$ integral over $B$.
    By Lemma \ref{lem:basechangeNormal},
    $C'\to B\otimes_A C'$ is a normal ring map essentially of finite presentation.
    By Lemma \ref{lem:ascendNormal},
    $B\otimes_A C'$ is normal, therefore the integral closure.
\end{proof}

The next result is about ascent of integral closure.
\begin{Lem}\label{lem:limitN2}
    Let $(A_i)_i$ be a direct system of integral domains and let $A=\colim_i A_i$.
    If the transition maps $A_i\to A_j$ are normal and
    essentially of finite presentation,
    and all $A_i$ are N-2,
    then $A$ is N-2.
\end{Lem}
\begin{proof}
    Let $K$ (resp. $K_i$) be the fraction field of $A$ (resp. $A_i$)
    and let $M/K$ be a finite extension.
    Then $M=M_i\otimes_{K_i}K$ for some finite extension $M_i/K_i$.
    Let $C_i$ be the integral closure of $A_i$ in $M_i$,
    so $C_i$ is finite over $A_i$.
    Now the ring $C_i\otimes_{A_i}A$ is the integral closure of $A$ in $M$
    by Corollary \ref{cor:NormalizationBaseChange}
    using that $A_i\to A_j$ is normal of finite presentation for all $j\geq i.$
\end{proof}

The next two results are about descent of integral closure.
\begin{Lem}\label{lem:indNormalDescendN-2}
    Let $A$ be an integral domain,
    $(B_i)_i$ a direct system of domains over $A$,
    and $B=\colim_i B_i$.
    Assume that $A\to B$ is faithfully flat,
    and that $A\to B_i$ is a normal ring map essentially of finite presentation for each $i$.
    Then if $B$ is N-2, so is $A$.
\end{Lem}
\begin{proof}
    Let $K,L$ be the fraction fields of $A,B$ respectively.
    As $A\to B$ is flat,
    $B_K:=K\otimes_A B$ is a localization of $B$ and thus an integral domain of fraction field $L$.

Let $M/K$ be a finite extension.
    The ring $M\otimes_A B=\colim_i M\otimes_A B_i$ is normal by the definition of normal ring map,
    and is finite flat over the integral domain $K\otimes_A B$.
    Therefore $B_M:=M\otimes_A B$ is a finite product of normal domains \citestacks{030C} whose fraction fields are finite over $L$.
    In particular, as $B$ is N-2,
    the integral closure of $B$ in $B_M$ is finite over $B$.

Let $C$ be the integral closure of $A$ in $M$.
Then $C\otimes_A B$ is the integral closure of $B$ in $B_M$ by Corollary \ref{cor:NormalizationBaseChange} applied to all $A\to B_i$.
Therefore $C\otimes_A B$ is finite over $B$,
so $C$ is finite over $A$ by descent \citestacks{03C4}.
\end{proof}

\begin{Lem}\label{lem:descendFiniteNormalization}
    Let $(A_i)_i$ be a direct system of integral domains and let $A=\colim_i A_i$.
    Let $K_i$ (resp. $K$) be the fraction field of $A_i$ (resp. $A$).
    Assume the transition maps $A_i\to A_j$ are faithfully flat for all $j\geq i$.

Let $i_0$ be an index, $M_{i_0}$ a
finite extension of $K_{i_0}$.
For all $i\geq i_0$ let 
$M_i=K_i\otimes_{K_{i_0}}M_{i_0}$.
Let $C_i$ be the integral closure of $A_i$ in $(M_i)_{\operatorname{red}}$.
Let $M=\colim_i M_i$, $C=\colim_i C_i$,
so $C$ is  the integral closure of $A$ in $M_{\operatorname{red}}$.
Then the following hold.
\begin{enumerate}[label=$(\roman*)$]
    \item\label{desc:red} There exists $i_1\geq i_0$ so that for all $i\geq i_1$ we have
    $M_{\operatorname{red}}=A\otimes_{A_{i}}(M_{i})_{\operatorname{red}}$.

    \item\label{desc:nor} If $C$ is finite over $A$,
    then there exists an $i_2\geq i_1$ so that for all $i\geq i_2$
    we have 
    $C=A\otimes_{A_{i}}C_{i}$
    and $C_i$ is finite over $A_i$.
\end{enumerate}
\end{Lem}
\begin{proof}
    For \ref{desc:red}, the ring $M$ is finite over $K$, so the nilradical of $M$ is finitely generated.
    It is clear that the nilradical of $M$ is the union of the nilradicals of $M_i$ for all $i\geq i_0$.
    Therefore we can find an $i_1$ so that the nilradical of $M_{i_1}$ contains a set of generators of the nilradical of $M$.
    It follows that the nilradical of $M_{i}$ contains a set of generators of the nilradical of $M$ for all $i\geq i_1$, so
    $M_{\operatorname{red}}=A\otimes_{A_{i}}(M_{i})_{\operatorname{red}}$ for all $i\geq i_1$.

For \ref{desc:nor},
it is clear that $C$ is the union of all $C_{i}$.
As $C$ is finite over $A$ we can find $i_2\geq i_1$ so that $C_{i_2}$ contains a set of generators of $C$ over $A$ as an algebra.
Then for all $i\geq i_2,$
$C_{i}$ contains a set of generators of $C$ over $A$ as an algebra,
so the canonical map $A\otimes_{A_{i}}C_{i}\to C$ is surjective.
By flatness,
$A\otimes_{A_{i}}C_{i}\to A\otimes_{A_{i}}(M_{i})_{\operatorname{red}}$ is injective.
Therefore \ref{desc:red} tells us
$A\otimes_{A_{i}}C_{i}\to C$ is also injective, showing $A\otimes_{A_{i}}C_{i}= C$.

Finally, $A_i\to A$ is also faithfully flat for all $i$ \citestacks{090N},
so $C_i$ is finite over $A_i$ by descent \citestacks{03C4}.
\end{proof}

\subsection{Valuation-theoretic characterization of the N-2 property}\label{subsec:valN2}
See  \cite[\S 1]{K-DVR-defectless} for defectless valuation rings.

For an extension $w/v$ of valuations with value groups $\Gamma_w\supseteq \Gamma_v$,
we have
the ramification index $e(w/v)=[\Gamma_w:\Gamma_v]$
and the initial index \[
\varepsilon(w/v)=\begin{cases}
    1 &\Gamma_w \text{ has no smallest positive element},\\
    [\bZ\rho:\bZ\rho\cap\Gamma_v] &\Gamma_w \text{ has a smallest positive element }\rho.
\end{cases}
\]
See \cite[Chap. VI, \S 8, no.4]{Bourbaki-CA-5-7}.
We also have the inertia degree $f(w/v)=[k_w:k_v]$,
where $k_w$ (resp. $k_v$) is the  residue field of $w$ (resp. $v$).

Let $v$ be a valuation on a field $K$.
For a finite extension $L/K$,
there are only finitely many extensions,
say $w_1,\ldots,w_g$,
of $v$ to $L$;
and we have the fundamental inequality
\[
[L:K]\geq \sum_{i=1}^g e(w_i/v)f(w_i/v). 
\]
The extension $L/K$ is \emph{defectless} if equality holds.
The valuation $v$ is \emph{defectless} if $L/K$ is defectless for every $L$.

We will use the following characterization of finiteness of integral closure, see \cite[Chap. VI, \S 8, no.5, Th. 2]{Bourbaki-CA-5-7}.
The integral closure of the valuation ring of $v$ in $L$ is finite if and only if $L/K$ is defectless and $\epsilon(w_i/v)=e(w_i/v)$ for all $i$.

\begin{Lem}\label{lem:ValN2}
    Let $V$ be a valuation ring with value group $\Gamma$.
    Then $V$ is N-2 if and only if the following hold.
    \begin{enumerate}
       [label=$(\roman*)$]
        \item\label{VRN2:defect} $V$ is defectless.
        \item\label{VRN2:vGP} Either $\Gamma$ is divisible
        or $\Gamma$ has a smallest positive element $\epsilon$ and $\Gamma/\bZ\epsilon$ is divisible\footnote{To be clear, we allow the case $\Gamma/\bZ\epsilon=0$,
        that is, $V$ is a DVR.}.
    \end{enumerate} 
\end{Lem} 
\begin{proof}
    Let $K$ be the fraction field of $V$.
    Let $v$ denote the valuation $K^\times\to\Gamma$.
    
    First, assume $\Gamma$ does not have a smallest positive element.
    Then the same is true for the value group of an extension of $V$ to a finite extension of $K$.
    Thus the initial index is always $1$.
    Therefore
    $V$ is N-2 if and only if $V$ is defectless and the ramification index is $1$ for all finite extensions (\emph{loc.cit.}),
    that is, the value group of $V$ is divisible.

    For the rest of the proof, assume $\Gamma$ has a smallest positive element $\epsilon$.
    If $\Gamma/\bZ\epsilon$ is not divisible,
    we can find $\gamma\in\Gamma$ and a prime number $p$ such that $\gamma\not\in p\Gamma+\bZ\epsilon$.
    Let $a$ be an element of $V$ with value $\gamma$
    and let $A=V[T]/(T^p-a)$.
    We know the value of $a$ is not in $p\Gamma+\bZ\epsilon$,
    in particular not in $p\Gamma$.
    Therefore $A$ is an integral domain and $v$ admits a unique extension $w$ to the fraction field of $A$
    with $e(w/v)=p$.
    Moreover, the value group of $w$ is $\bZ\frac{1}{p}\gamma +\Gamma\subseteq \frac{1}{p}\Gamma$.
    Let $\delta$ be a positive element of this group.
    If $\delta<\epsilon$,
    we see $p\delta<p\epsilon$.
    Note that $p\delta\in\Gamma$.
Let $m\in\bZ_{\geq 0}$ be largest so that $p\delta\geq m\epsilon$, so $m<p$.
Then $0\leq p\delta-m\epsilon<\epsilon$,
so $p\delta=m\epsilon$.
As $p\delta>0$ we have $m\in\{1,2,\ldots,p-1\}$.
    Therefore $\frac{1}{p}\epsilon\in \bZ\frac{1}{p}\gamma +\Gamma$,
    so $\frac{1}{p}\gamma\in \bZ\frac{1}{p}\epsilon +\Gamma$,
    contradicting the choice of $\gamma$.
    Therefore $\epsilon$ is the smallest positive element of $\bZ\frac{1}{p}\gamma +\Gamma$.
    Thus 
    $\varepsilon(w/v)=1$,
    whereas $e(w/v)=p$,
    so the integral closure of $V$ in the fraction field of $A$ is not finite (\emph{loc.cit.});
    thus $V$ is not N-2.
    
    On the other hand, if $\Gamma/\bZ\epsilon$ is divisible,
    then for every finite index extension $\Gamma\subseteq\Gamma'$ we have $\Gamma'=\bQ\epsilon\cap \Gamma'+\Gamma$,
    and $\frac{1}{n}\epsilon$ (where $n=[\Gamma':\Gamma]$) is the smallest positive  element of $\Gamma'$.
    Therefore the initial index
    is always equal to the ramification index for a finite extension of the valued field $K$.
    Thus $V$ is N-2 if and only if $V$ is defectless (\emph{loc.cit.}).
\end{proof}

\begin{Cor}\label{cor:VRN2}
    Let $V$ be a valuation ring.
    Then the following are equivalent.
    \begin{enumerate}
       [label=$(\roman*)$]
        \item\label{VRN2c:N2} $V$ is N-2.
        \item\label{VRN2c:uN2} $V(x_1,\ldots,x_n)$ is N-2 for all $n\in\bZ_{\geq 0}$. 
        \item\label{VRN2c:sN2} $V(x_1,\ldots,x_n)$ is N-2 for some $n\in\bZ_{>0}$. 
    \end{enumerate}
\end{Cor} 
\begin{proof}
    The ring $V(x_1,\ldots,x_n)$ is a valuation ring with the same value group as $V$.
    If $V$ is N-2,
    then $V$ is defectless by Lemma \ref{lem:ValN2}, so $V(x_1,\ldots,x_n)$ is defectless by \cite[Theorem 1.1]{K-DVR-defectless},
    thus is N-2 by Lemma \ref{lem:ValN2} again.
    On the other hand,
    assume $V(x_1,\ldots,x_n)$ is N-2 for some $n>0$.
    Then $V$ is N-2 by Lemma \ref{lem:indNormalDescendN-2}.
\end{proof}

\begin{Cor}\label{cor:valN2hs}
    Let $V$ be a valuation ring.
    Then $V$ is N-2 if and only if the Henselization of $V$ is N-2.
\end{Cor}
\begin{proof}
    Same proof as Corollary \ref{cor:VRN2},  except that 
    defectlessness is easier \cite[Theorem 2.14]{K-DVR-defectless}.
\end{proof}

\subsection{Naive dualizing modules for certain finitely presented algebras over Pr\"ufer domains}\label{subsec:S2ify}

\begin{Lem}\label{lem:HomtoS2}
    Let $A$ be a ring, $M$ an $A$-module, $a,b\in A$.
    If $M$ has zero $a$-torsion and $M/aM$ has zero $b$-torsion,
    then the same is true for $\Hom_A(N,M)$
    for all $A$-modules $N$.
\end{Lem}
\begin{proof}
    The exact sequence
    \[
    \begin{CD}
        0@>>> M@>{a}>> M@>>> M/aM
    \end{CD}
    \]
    induces an exact sequence
    \[
    \begin{CD}
        0@>>> \Hom_A(N,M)@>{a}>> \Hom_A(N,M)@>>> \Hom_A(N,M/aM),
    \end{CD}
    \]
    showing that $\Hom_A(N,M)$ has zero $a$-torsion,
    and that $\Hom_A(N,M)/a\Hom_A(N,M)$
    is a submodule of $\Hom_A(N,M/aM)$,
    which has zero $b$-torsion.
\end{proof}
\begin{Thm}\label{thm:S2ify}
    Let $D$ be an integral domain,
    $P$ a polynomial $D$-algebra
    of finitely many variables,
    and $B$ an integral domain containing and finite over $P$.
    Let $\omega=\Hom_{P}(B,P)$
    and let $B'=\Hom_B(\omega,\omega)$.
    Then the following hold.
    \begin{enumerate}[label=$(\roman*)$]
        \item\label{S2ify:S2} For every primitive polynomial $f\in P$,
        $B'$ has zero $f$-torsion and
        $B'/fB'$ is torsion-free over $D$.
        \item\label{S2ify:finite} If $D$ is Pr\"ufer,
        then $B'$ is a finite and finitely presented $B$-algebra in the fraction field of $B$;
        the construction of $\omega$ and $B'$ is compatible with flat base change to an integral domain $E$ containing $D$.
    \end{enumerate}
\end{Thm}
\begin{proof}
    A primitive polynomial $f$ is a nonzerodivisor in the polynomial ring $P$ over (any ring) $D$ and $P/fP$ is flat over $D$,
    see \cite[Corollary to Theorem 22.6]{Mat-CRT},
    which is the Noetherian case, and the general case follows from taking a direct union of subrings of finite type over $\bZ$.
    Now
    Lemma \ref{lem:HomtoS2} gives \ref{S2ify:S2}.

     For \ref{S2ify:finite},
     since $D$ is Pr\"ufer,
     the torsion-free $D$-algebra $B$ is flat over $D$, % \cite[p. 25]{Coherent-Rings},
     therefore finitely presented as a $P$-module \citestacks{053G}.
     Moreover, $P$ is coherent \cite[Theorem 7.3.3]{Coherent-Rings}.
     Therefore $\omega$ is a finitely presented $P$-module by \cite[Corollary 2.5.3]{Coherent-Rings};
     note that finitely presented and coherent are the same for a module over a coherent ring,
     see \citestacks{05CX}.
     By \citestacks{0561}, $\omega$ is a finitely presented $B$-module,
     so $B'$ is a finitely presented $B$-module by \cite[Corollary 2.5.3]{Coherent-Rings} again.
     Flat base change now follows from \citestacks{087R}.

     Denote by $K$ and $L$ the fraction fields of $P,B$ respectively.
     We know $\omega$ (and thus $B'$) is a torsion-free $P$-module,
     and $\omega\otimes_P K=\Hom_K(L,K)$,
     which is an $L$-vector space of dimension $1$.
    Therefore $\omega$ is a $B$-module of rank $1$,
    so $B'\otimes_B L=L$,
    showing that the multiplication on $B'$ is commutative and that $B'$ and $B$ have the same fraction field.
      \end{proof}

\subsection{An elementary version of Serre's criterion for normality}\label{subsec:Serre}
\begin{Lem}\label{lem:Serre}
    Let $A$ be an integral domain,
    $S,T$ two multiplicative subsets of $A$.
    Assume
    \begin{enumerate}[label=$(\roman*)$]
        \item $S^{-1}A$ is normal.
        \item $T^{-1}A$ is normal.
        \item For all $s\in S$ and $t\in T$, $A/sA$ has zero $t$-torsion.
    \end{enumerate}
    Then $A$ is normal.
\end{Lem}
\begin{proof}
    Let $y$ be an element in the fraction field of $A$ integral over $A$.
    By our assumptions,
    there exist $s\in S,t\in T$,
    such that $sy,ty\in A$.
    Then $tsy=sty\in sA$,
    so $sy\in sA$ as $A/sA$ has zero $t$-torsion. 
    As $A$ is an integral domain, $y\in A$.
\end{proof}
\begin{Cor}\label{cor:PruferSerre}
     Let $D$ be a Pr\"ufer domain,
    $P$ a polynomial $D$-algebra
    of finitely many variables,
    and $B$ an integral domain containing and finite over $P$.

    Assume that $B_{\fm P}$ is normal for all maximal ideals $\fm$ of $D$.
    Then the following hold.
    \begin{enumerate}[label=$(\roman*)$]
        \item\label{corSerre:S2isnor} If $B\otimes_D F$ is normal, where $F$ is the fraction field of $D$,
    then the algebra $B'$ as in Theorem \ref{thm:S2ify} is normal.
    \item\label{corSerre:N1} $B$ is N-1.
    \end{enumerate}
\end{Cor}
\begin{proof}
For \ref{corSerre:N1},
    as the field $F$ is universally Japanese \citestacks{0335}, $B\otimes_D F$ is N-1,
    so we may replace $B$ by a finite algebra to make $B\otimes_D F$ normal.
    Thus we only need to prove \ref{corSerre:S2isnor}.

    As the construction of $B'$ as in Theorem \ref{thm:S2ify} is compatible with localization,
    we may assume $D$ local.
    Then Lemma \ref{lem:Serre} applies with $A=B', S$ the set of primitive polynomials in $P$,
    and $T$ the set of nonzero elements in $D$.
\end{proof}

The next two corollaries are used in the proof of Theorem \ref{thm:smallPrUJ} and are interesting on their own.
\begin{Cor}\label{cor:largeN1smallN1}
    Let $D$ be a Pr\"ufer domain,
    $P$ a polynomial $D$-algebra
    of finitely many variables,
    and $B\subseteq C$ integral domains containing and finite over $P$.
    If $C$ is N-1, so is $B$. 
\end{Cor}
\begin{proof}
    We may assume $C$ is normal.
    Let $B'$ be the normalization of $B$,
    so we have $P\subseteq B\subseteq B'\subseteq C$.
    Note that this does not immediately give $B$ is N-1 as $P$ is not necessarily Noetherian.

    As $C$ is finite over $P$,
    it is finite over $B'$.
    Also $C_{\fm P}$ is flat over the Pr\"ufer domain $B'_{\fm P}$ for every $\fm\in\Max(D)$.
    As $B'$ is an integral domain,
    the flat locus $U\subseteq \Spec(C)$ of $B'\to C$ is open,
    see    \citestacks{05IM}.
    We know $U$ contains $\Spec(C_{\fm P})\subseteq\Spec(C)$,
   therefore
    we have in $\Spec(C)$
    \[
    \bigcap_{f\in P,f\not\in\fm P} D(f)\subseteq U. %'\subseteq U
    \]
    As the constructible topology is compact \citestacks{0901},
    we have $D(f)\subseteq U$ for some $f=f(\fm)\in P,f\not\in\fm P$.
    Then $C_f$ is flat, thus of finite presentation over $B'_f$ by \citestacks{053G}.
    By \citestacks{0367} $B'_f$ is finite over $P_f$,
    so there exists a finite $B$-algebra $B''=B''(\fm)\subseteq B'$
    such that $B''_f= B'_f$.
    By Lemma \ref{lem:GenericsQC} below there exist finitely many $\fm_1,\ldots,\fm_t\in\Max(D)$
    such that $D(f(\fm_1)),\ldots,D(f(\fm_t))$ cover $\{\fm P\mid \fm\in\Max(D)\}$,
    and we replace $B$ be the finite $B$-algebra generated by all $B''(\fm_i)$'s.
    Then $B_{\fm P}$ is normal for all $\fm\in\Max(D)$,
    so $B$ is N-1 by Corollary \ref{cor:PruferSerre}.
\end{proof}
We used
\begin{Lem}\label{lem:GenericsQC}
    Let $R$ be a ring and $P$ a polynomial $R$-algebra.
    Then the sets
    \begin{align*}        
    \Sigma&=\{\fp P\mid \fp\in\Spec(R)\}\subseteq\Spec(P)\text{ and}\\
    \mathrm{M}&=\{\fm P\mid \fm\in\Max(R)\}\subseteq\Spec(P)
    \end{align*}
 are quasi-compact.
\end{Lem}
\begin{proof}
    Indeed, the canonical continuous bijection $\Sigma\to \Spec(R)$
    is a homeomorphism.
    To see this, let $F\in P$.
    Then the image of $D(F)\cap \Sigma$ in $\Spec(R)$ is the (finite) union of principal opens defined by the coefficients of $F$, thus is open, as desired.
    The lemma follows as $\Spec(R)$ and $\Max(R)$ are quasi-compact. 
\end{proof}

The following result is used in the proof of Theorem \ref{thm:smallPrUJ}.
\begin{Cor}\label{cor:unionN2allN2}
    Let $(D_i)_i$ be a direct system of Pr\"ufer domains such that the transition maps $D_i\to D_j$ are injective and finite.
    Let $P_0$ be a polynomial algebra of finitely many variables over $\bZ$ and let $P_i=D_i\otimes_\bZ P_0$.
    If $P:=\bigcup_i P_i$ is N-2,
    then $P_i$ is N-2 for all $i$.
\end{Cor}
Note that
the polynomial algebra in the statement becomes redundant
once Theorem \ref{thm:smallPrUJ} is proved.
\begin{proof}
    Let $i_0$ be a given index and let $B_{i_0}$ be an integral domain containing and finite over $P_{i_0}$.
    We must show $B_{i_0}$ is N-1.

    Let $M_{i_0}$ be the fraction field of $B_{i_0}$,
    and let $M_i=P_i\otimes_{P_{i_0}}M_{i_0}$
    and $M=P\otimes_{P_{i_0}}M_{i_0}$.
    As $P$ is N-2, the integral closure of $P$ in $M_{\operatorname{red}}$ is finite over $P$.

    As the maps $D_i\to D_j$ are injective and finite,
    and as each $D_i$ is a Pr\"ufer domain,
    the maps $D_i\to D_j$ are  faithfully flat,
    hence so are the maps $P_i\to P_j$.
    Therefore Lemma \ref{lem:descendFiniteNormalization} applies and tells us there exists an $i\geq i_0$ such that
    the integral closure $C_i$ of $P_i$ in $(M_i)_{\operatorname{red}}$ is finite over $P_i$ and therefore finite over $P_{i_0}$.
    As $(M_i)_{\operatorname{red}}$ is a finite product of fields and a localization of $C_i$
it follows that $C_i$ is a finite product of normal domains \citestacks{030C}.
Applying Corollary \ref{cor:largeN1smallN1} to a normal domain factor of $C_i$,
we see $B_{i_0}$ is N-1,
 as desired.
\end{proof}

\section{Proof of Theorems}

\subsection{Quick reductions}

\begin{Lem}\label{lem:valN2quot}
    Let $D$ be a Pr\"ufer domain that is N-2.
    Then for all $\fp\in\Spec(D)$,
    $D_\fp$ and $D/\fp$ are N-2.
\end{Lem}
\begin{proof}
    By \citestacks{032G} $D_\fp$ is N-2.
    To show $D/\fp$ is N-2,
    note that a finite extension  $M/\kappa(\fp)$ is realized by a finite extension $L$ of the fraction field $K$ of $D$
    of the same degree.
To be precise, when $M/\kappa(\fp)$ is generated by a single element,
we have $M=\kappa(\fp)[T]/\overline{F}(T)$ for some monic irreducible polynomial $\overline{F}\in \kappa(\fp)[T]$.
We can lift $\overline{F}$ to a monic polynomial $F\in D_\fp[T]$.
As $D_\fp$ is a valuation ring we see $F$ is irreducible,
$L=K[T]/F(T)$
is a finite extension, $D_\fp$ has a unique extension to $L$,
and the corresponding residue field extension is $M/\kappa(\fp)$ for degree reasons.
In general, we can apply this construction inductively,
and we get a finite extension $L/K$
so that $D_\fp$ has a unique extension to $L$
and that the corresponding residue field extension is $M/\kappa(\fp)$.

    The integral closure $E$ of $D$ in $L$ is a Pr\"ufer domain finite over $D$,
    and the integral closure $E_1$ of $D/\fp$ in $M$ is the image of $E$,
    as the quotient of a Pr\"ufer domain by a prime ideal is a Pr\"ufer domain and therefore normal.
    Thus $E_1$ is finite over $D/\fp$.
\end{proof}

\begin{Lem}\label{lem:redtoPoly}
    Let $D$ be a Pr\"ufer domain.
    Then $D$ is universally Japanese if and only if for all prime ideals $\fp$ of $D$,
    every polynomial algebra $P$ of finitely many variables  over $D/\fp$ is N-2. 
\end{Lem}
\begin{proof}
    ``Only if'' is trivial.
    To show ``if,''
    let $A$ be a finitely generated $D$-algebra that is an integral domain.
    We need to show $A$ is N-1.
    We may assume $D$ is contained in $A$,
    as the quotient of a Pr\"ufer domain by a prime ideal is a Pr\"ufer domain.
    Let $F$ be the fraction field of $D$.

    Let $\fM$ be a maximal ideal of $A$ and let $\fp=\fM\cap D$.
    By \citestacks{032H},
    it suffices to show $A_f$ is N-1 for some $f\in A,f\not\in\fM$.
    Let $n$ be the relative dimension of $A/D$ at $\fM$,
    that is,
    $n=\dim A_\fM/\fp A_\fM$.
    By \citestacks{00QE},
    replacing $A$ by a localization,
    there exists a polynomial $D$-algebra $P$ of $n$ variables and a quasi-finite ring map $P\to A$.
    Notice that $A\otimes_D \kappa(\fp)$ is equidimensional and its dimension is equal to $\dim(A\otimes_D F)$,
    see \citestacks{00QK}.
    Therefore $P\to A$ is injective.
    By \citestacks{00QB},
    there exists a finite $P$-algebra $B$ in $A$ such that $\Spec(A)\to\Spec(B)$ is an open immersion.
    Since $P$ is N-2, we see $B$ is N-1,
    so $A$ is N-1, as desired.
\end{proof}

\subsection{Proof of Theorem \ref{thm:ValRingUJ}}
Strictly speaking, this is unnecessary, as the nontrivial implication in Theorem \ref{thm:ValRingUJ} is a special case of Theorem \ref{thm:smallPrUJ}.
However, as the proof is almost immediate at this point, we do present it.

The implication \ref{VRUJ:uJ} implies \ref{VRUJ:N2} is trivial,
whereas \ref{VRUJ:N2}\ref{VRUJ:uN2}\ref{VRUJ:sN2}
are equivalent
    by Corollary \ref{cor:VRN2}.
    It remains to show \ref{VRUJ:uN2} implies \ref{VRUJ:uJ}.
    
    By Lemmas \ref{lem:valN2quot}
    and \ref{lem:redtoPoly},
    it suffices to show, for a valuation ring $V$ and $n\in\bZ_{>0}$,
    if $V(x_1,\ldots,x_n)$ is N-2,
    then $P:=V[x_1,\ldots,x_n]$ is also N-2.
    Let $B$ be an integral domain containing and finite over $P$.
    We must show $B$ is N-1.

    Let $\fm$ be the maximal ideal of $V$. 
    We have $P_{\fm P}=V(x_1,\ldots,x_n)$ is N-2,
    so $B_{\fm P}$ is N-1.
    Replacing $B$ by a finite algebra in its fraction field,
    we may assume $B_{\fm P}$ is normal.
    Then $B$ is N-1 by Corollary \ref{cor:PruferSerre}.
\subsection{Proof of Theorem \ref{thm:PrufUJ}}
We know \ref{PrUJ:uJ} implies \ref{PrUJ:N2} and \ref{PrUJ:locN2} implies \ref{PrUJ:maxN2} trivially,
whereas \ref{PrUJ:N2} implies \ref{PrUJ:locN2} by 
Lemma \ref{lem:valN2quot}.
Therefore,
it suffices to show a Pr\"ufer domain $D$ is universally Japanese
assuming for all $\fm\in\Max(D)$, $D_\fm$ is N-2 and has divisible value group and perfect residue field.

    As the assumptions are inherited by quotients (Lemma \ref{lem:valN2quot}),
    by Lemma \ref{lem:redtoPoly}
    it suffices to show $P:=D[x_1,\ldots,x_n]$ is N-2.
    Let $B$ be an integral domain containing and finite over $P$.
    We must show $B$ is N-1.
    Let $L$ be the fraction field of $B$.

    Let $\fm\in\Max(D)$.
    %and consider the prime ideal $\fm P$ of $P$.
    Then $P_{\fm P}=D_\fm(x_1,\ldots,x_n)$
    is a valuation ring that is N-2 (Corollary \ref{cor:VRN2}).
    Therefore the integral closure of $P_{\fm P}$ in $L$ is finite over $P_{\fm P}$.
    We can then find a finite $B$-algebra
    $C=C(\fm)$ in $L$ such that $C_{\fm P}$ is normal.

Let $\fq$ be a minimal prime of $\fm C$, so $\fq\cap P=\fm P$ as $P$ is normal and $C$ is an integral domain (going-down, \citestacks{00H8}).
    %Assume for the moment that 
    Therefore $C_\fq$ is normal, so it is a valuation ring extending $P_{\fm P}$.
    Since the value group of $P_{\fm P}$ is divisible,
    we see $\fm C_\fq=\fq C_\fq$.
    Therefore $C_\fq/\fm C_\fq$ is a field extension of $D/\fm$,
    and it is separable\footnote{Here this just means $C_\fq/\fm C_\fq$ is finite separable over a purely transcendental extension of $D/\fm$.}
    as $D/\fm$ is perfect by assumption.
    As $D$ is Pr\"ufer, $D\to C$ is flat,
    and therefore of finite presentation \citestacks{053G},
    therefore by \citestacks{00TF} $D\to C_g$ is smooth for some $g\in C,g\not\in\fq$.
    In particular, $C_g$ is normal (Lemma \ref{lem:ascendNormal}).
    
    From the previous discussion there exists an open $U=U(\fm)$ of $\Spec(C)$ containing all minimal primes of $\fm C$, that is, all preimages of $\fm P$, such that $U$ is normal.
    The image in $\Spec(P)$  of the complement of $U$ is closed (cf. \citestacks{01WM})
    and does not contain $\fm P$,
    so there exists $f=f(\fm)\in P,f\not\in\fm P$
    such that $C_f$ is normal.

    By Lemma \ref{lem:GenericsQC},
    there exist finitely many $\fm_1,\ldots,\fm_t\in\Max(D)$
    such that $D(f(\fm_1)),\ldots,D(f(\fm_t))$ covers $\{\fm P\mid \fm\in\Max(D)\}$,
    and we replace $B$ be the finite $B$-algebra generated by all $C(\fm_i)$'s.
    Now for every $\fm\in\Max(D)$, $B_{\fm P}$ is normal.
    By Corollary \ref{cor:PruferSerre},
    $B$ is N-1.
\subsection{Proof of Theorem \ref{thm:ExampleUJ}}
    Lemma \ref{lem:ValN2}
    show that the rings in \ref{ExUJ:absVal} and \ref{ExUJ:char0Val}
    are N-2,
    as a valuation ring of residue characteristic zero is defectless, see \cite[Corollary 2.12]{K-DVR-defectless}.
    Either
    Theorem \ref{thm:ValRingUJ} or Theorem \ref{thm:PrufUJ}
    then gives \ref{ExUJ:absVal} and \ref{ExUJ:char0Val}.
 
    By Theorem \ref{thm:PrufUJ},
    \ref{ExUJ:absVal} implies \ref{ExUJ:absPr} and
    \ref{ExUJ:char0Val} implies \ref{ExUJ:char0Pr}.
    The rings in \ref{ExUJ:absclosVal}\ref{ExUJ:absclosPr}\ref{ExUJ:absclosDdkd}
    are absolutely integrally closed Pr\"ufer domains.
    Finally, the normalization of a Noetherian integral domain of dimension $1$
    is a Dedekind domain by the theorem of Krull-Akizuki,
    see \citestacks{09IG},
    so \ref{ExUJ:absclosDdkd} gives \ref{ExUJ:absclosNoe1}.
\subsection{Proof of Theorem \ref{thm:smallPrUJ}}
Again, as universally Japanese implies N-2 by definition, it suffices to show an N-2 Pr\"ufer domain is universally Japanese.

    Let $D$ be a Pr\"ufer domain that is N-2.
    By Lemmas \ref{lem:valN2quot} and \ref{lem:redtoPoly},
    it suffices to show $P:=D[x_1,\ldots,x_n]$ is N-2.

Let $F$ be the fraction field of $D$.
Let $\overline{F}$ be an algebraic closure of $F$.
Let $(F_i)_i$ be the family of all finite extensions of $F$ inside $\overline{F}$.
Let $D_i$  be the integral closure of  $D$ in $F_i$.
Let $P_i=D_i\otimes_D P$,
    a polynomial algebra over $D_i$,
    $K_i$ its fraction field.

As $\bigcup_i D_i$ is an absolutely integrally closed Pr\"ufer domain,
it is universally Japanese by Theorem \ref{thm:ExampleUJ}\ref{ExUJ:absPr},
so $\bigcup_i P_i$
is N-2.
As $D$ is N-2, all $D_i$ are finite over $D$.
By Corollary \ref{cor:unionN2allN2},
all $P_i$ are N-2, in particular $P$ is N-2.

\subsection{Proof of Theorem \ref{thm:absIntNotUJ}}
By \citestacks{032G} it suffices to show that for a Nagata local domain $R$ of equal characteristic zero and dimension $2$,
the absolute integral closure $R^+$ of $R$ is not universally Japanese.
Let $x,y$ be a system of parameters of $R$.
Let $I=(x,y)\subseteq R^+$.
We claim that the Rees algebra
$S=R^+[IT]=\bigoplus_n I^nT^n\subseteq R^+[T]$
is not N-1.
This tells us $R^+$ is not universally Japanese.

Denote by $\overline{J}$ the integral closure of an ideal $J$ in the normal domain $R^+$.
Then the normalization of $S$ is
$S'=\bigoplus_n \overline{I^n}T^n.$
If $S'$ is finite over $S$,
then $\overline{I}/I$, and therefore $\overline{I}$, is finitely generated.
Therefore it suffices to show $\overline{I}$
is not finitely generated.

Assume $\overline{I}$
is finitely generated.
Let $R'$ be a finite $R$-algebra 
inside $R^+$ so that $\overline{I}=(\overline{I}\cap R')R^+$.
We may assume $R'$ is normal as 
$R$ is Nagata.
We replace $R$ with the localization of  $R'$ at a maximal ideal.
We have reduced to the case
$\overline{I}\subseteq \fm R^+$,
where $\fm$ is the maximal ideal of the normal local ring $R$, and we will derive a contradiction.

Let $k$ be a coefficient field of the completion $R^\wedge$,
so the canonical map $a:k[[x,y]]\to R^\wedge$ is finite.
As $R$ is Nagata $R^\wedge$ is reduced \citestacks{0331}.
After a linear change of coordinates ($\bQ$ is infinite!),
we may assume $a$ is \'etale at $(x-y),(x+y)\in\Spec(k[[x,y]])$.

Consider the ring $R'=R[T]/(T^2-x^2+y^2)$, so $R'^\wedge=R^\wedge[T]/(T^2-x^2+y^2)$.
As $R$ is $2$-dimensional and normal, it is Cohen-Macaulay,
thus so is $R'$.
Since $R'$ is \'etale over $R$ on $D(x-y)\cap D(x+y)$ and $R'^\wedge$ is \'etale over $k[[x,y]][T]/(T^2-x^2+y^2)$ at $\sqrt{(x-y)}$ and $\sqrt{(x+y)}$ we see $R'$ is $(R_1)$ and therefore normal.

The element $t\in R'$, image of $T$,
satisfies $t^2=x^2-y^2$.
Therefore (for any embedding $R'\subseteq  R^+$) $t\in \overline{I}\cap R'\subseteq \fm R^+\cap R'$.
However, the normal ring $R'$ of equal characteristic zero is a splinter (cf. \cite[Lemma 2]{Hoc73-splinter}),
so $\fm R^+\cap R'=\fm R'$,
giving $t\in\fm R'$, a contradiction.

\begin{Rem}
    The Rees algebra $S$ used here is actually of finite presentation over $R^+$, as opposed to the situation in Example \ref{exam:infVarsNotUJ}.
    Indeed, as a normal Noetherian ring of dimension $2$ is Cohen--Macaulay,
    $R^+$ is flat over $\bQ[x,y]$,
    so $S$ is a flat base change of the Rees algebra of $(x,y)$ over $\bQ[x,y]$.
    This also implies that $S$ has finite Tor dimension over $R^+$. 
\end{Rem}

\section{Examples}\label{sec:ex}
In \cite{Heitmann-PID-locEx-notEx} a PID $R$ which is locally universally Japanese but not universally Japanese is constructed,
exhibiting failure of Theorem \ref{thm:PrufUJ} in general.
The fraction field of $R$ has characteristic $p$,
which is necessary,
as every Dedekind domain with characteric zero fraction field
is universally Japanese \citestacks{0335}.

For non-Noetherian rings, similar failure can occur in any characteristic.
\begin{Exam}
    Let $V$ be a DVR that is N-2 (equivalent to universally Japanese by Theorem \ref{thm:ValRingUJ}).
    We construct an integral domain $D$ containing $V$ with the property that $D_\fm$ is a universally Japanese DVR
    dominating $V$
    for all $\fm\in\Max(D)$,
    yet $D$ is not N-2.
    
    Let $x_1$ be a uniformizer of $V$.   
    Let $R_0=V$
    and let $R_{i+1}=R_i[X_{i+1},Y_i]/(x_i-X_{i+1}Y_{i}^{2})$,
    a flat $R_i$-algebra,
    where $x_i$ is the image of $X_i$ in $R_i$.
    Then 
    \[
R_{i+1}=V[X_{i+1},Y_1,\ldots,Y_i]/(x_1-X_{i+1}Y_{i}^{2}\ldots Y_2^{2}Y_1^2),
    \]
    and we denote by $y_i$ the image of $Y_i$ in $R_{i+1}$.
    Let $D_{i+1}$ be the localization of $R_{i+1}$ with respect to the complement of $x_{i+1}R_{i+1}\cup y_1R_{i+1}\cup\ldots\cup y_iR_{i+1}$,
    so $D_{i+1}$ is a semilocal PID with maximal ideals
    $x_{i+1}D_{i+1}, y_1D_{i+1},\ldots, y_iD_{i+1}$.
    We also note that $y_jR_{i+1}\cap R_i=y_jR_i$ when $j\leq i$,
    and $y_{i+1}R_{i+1}\cap R_i=x_iR_i=x_{i+1}R_{i+1}\cap R_i$,
    as they are all primes of height $1$ with obvious inclusion relations.
    Therefore the inclusion $R_i\to R_{i+1}$ gives an inclusion $D_i\to D_{i+1}$,
    and we let $D=\bigcup_i D_i$.
    We verify $D$ is the desired example in several steps.

    Write $V_{\infty i}=(D_i)_{x_i D_i}$ and 
    $V_{ji}=(D_i)_{y_j D_i}$ for all $j<i$.
    They are N-2 DVRs as $V$ is universally Japanese.

\begin{SteplocUJnotUJ}\label{step:locUJnotUJ-maxD}
Write $\fm_\infty = (x_1,x_2,\ldots)\subseteq D$
    and $\fm_j=y_jD$.
Then $\Max(D)=\{\fm_j\mid j\in\bZ_{>0}\cup\{\infty\}\}$.
Moreover,
\begin{equation*}
  \fm_j\cap D_i=\begin{cases}
    y_jD_i, & \text{if $j<i$}.\\
    x_iD_i, & \text{otherwise}.
  \end{cases}
\end{equation*}
%
%Then $\fm=y_i D$ for some $i$,
%or $\fm=(x_1,x_2,\ldots)$.

To see this, clearly all $\fm_j$ are maximal ideals.
Let $\fm$ be a maximal ideal of $D$.
For an $i$ so that $\fm\cap D_i\neq 0$, $\fm\cap D_i$ is either $y_j D_i$ for some $j<i$ or $x_i D_i$.
In the first case, $\fm=\fm_j$.
If the second case holds for all large $i$, then $\fm=\fm_\infty$.
\end{SteplocUJnotUJ}

\begin{SteplocUJnotUJ}
    $D_{\fm}$ is a universally Japanese DVR for all $\fm\in\Max(D)$.

If $\fm=\fm_\infty$,
then $D_{\fm}=\bigcup_i V_{\infty i}$,
and a uniformizer of $V_{\infty i}$ goes to a uniformizer of $V_{\infty, i+1}$.
Therefore $D_\fm$ is a DVR.
Moreover, $V_{\infty i}\to V_{\infty, i+1}$ is essentially smooth,
 as $V_{\infty, i+1}\cong V_{\infty i}(Y_i)$.
Therefore $D_\fm$ is N-2 by Lemma \ref{lem:limitN2}.
On the other hand, if $\fm=\fm_j$ for some $j$,
then $D_{\fm}=\bigcup_{i>j} V_{ji}$
and the rest of the proof is the same.
\end{SteplocUJnotUJ}

\begin{SteplocUJnotUJ}
    $D$ is not N-2.

    Indeed, we show that the ring $A=D[T]/(T^n-x_1)$ is not N-1 for all $n\in\bZ_{>1}$.
    Note that $A$ is an integral domain as $x_1$ is a uniformizer of $D_{\fm_\infty}$.
    Let $t\in A$ be the image of $T$.
    Then $A_{\fm_\infty}$ is a DVR with uniformizer $t$.
    On the other hand, for $j\neq\infty$, $A_{\fm_j}=D_{\fm_j}[T]/(T^n-uy_j^{2})$ for some $u\in D_{\fm_j}^\times$.
    This tells us $A_{\fm_j}$ is not normal.
    Indeed, $A_{\fm_j}$ is finite over $D_{\fm_j}$, and $A_{\fm_j}/y_j A_{\fm_j}=(D_{\fm_j}/y_jD_{\fm_j})[T]/(T^n)$ is local,
    so $A_{\fm_j}$ is local;
    therefore it is the quotient of the regular local ring $B:=D_{\fm_j}[T]_{(y_j,T)}$ by the element $T^n-uy_j^{2}$.
    As $T^n-uy_j^{2}$ is in the square of the maximal ideal of $B$ we see $A_{\fm_j}$ is not regular (see \cite[Theorem 14.2]{Mat-CRT}),
    so $A_{\fm_j}$ is not a DVR,
    therefore not normal.

    If the normalization $A'$ of $A$ were finite over $A$,
    then the fact $A_{\fm_\infty}$ is normal tells us $A_{\fm_\infty}=A'_{\fm_\infty}$,
    so $A_{f}=A'_{f}$ for some $f\in D$,
    $f\not\in\fm_\infty$.
    As $A_{\fm_j}$ is not normal we conclude that $f\in\fm_j$ for all $j$.
    This is impossible as if $f\in D_k$,
    then $f\not\in \fm_\infty\cap D_k=x_k D_k=\fm_l\cap D_k$ for all $l\geq k$,
    so $f\not\in \fm_l$ for all $l\geq k$.
\end{SteplocUJnotUJ}
\end{Exam}

In this example, the element $x_1$ has valuation $1$ at $\fm_\infty$ and valuation $2$ at other $\fm_j$,
creating problems.
Indeed we have
\begin{Prop}
    Let $D$ be a Pr\"ufer domain containing $\bQ$.
    Assume that there exists $y\in D$ such that $yD_\fm=\fm D_\fm$ for all $\fm\in\Max(D)$.
    Then $D$ is universally Japanese
    if and only if $D_\fm$ is N-2 for all $\fm\in\Max(D)$.
\end{Prop}
The rings considered in \cite[\S 3]{Loper-Lucas-Almost-Dedekind-branches} are covered by this result if the characteristic of $K$ is zero.
\begin{proof}
If $D$ is universally Japanese, then trivially $D$ is N-2,
so $D_\fm$ is N-2 for all $\fm\in\Max(D)$
by Lemma \ref{lem:valN2quot}.
Conversely,
assume $D_\fm$ is N-2 for every $\fm\in\Max(D)$,
we will show $D$ is universally Japanese.
By Theorem \ref{thm:smallPrUJ},
it suffices to show $D$ is N-2.

We aim to apply Corollary \ref{cor:unionN2allN2}.
    Let $\overline{F}$ be an algebraic closure of the fraction field $F$ of $D$.
    Inductively choose elements $y_i\in\overline{F}$ so that $y_0=y$ and $y_{n}^{n}=y_{n-1}$.
    Write $D_n=D[y_n]\subseteq \overline{F}$.
    We see inductively that
    $D_n$ is Pr\"ufer, $y_n(D_n)_\fn=\fn (D_n)_\fn$ for all $\fn\in\Max(D_n)$,
    and $\Max(D_n)\to\Max(D)$ is bijective.

    Let $D_\infty=\bigcup_n D_n$,
    so $D_\infty$ is Pr\"ufer and $\Max(D_\infty)\to\Max(D)$ is bijective.
    By Lemma \ref{lem:ValN2} the value groups of the local rings of $D_\infty$ at maximal ideals are divisible.
    By Theorem \ref{thm:ExampleUJ}\ref{ExUJ:char0Pr},
    $D_\infty$ is N-2.
    As every $D_n$ is finite over $D$ by construction,
    Corollary \ref{cor:unionN2allN2} tells us every $D_n$,
    in particular $D$, is N-2, as desired.
\end{proof}

We include another example.
\begin{Exam}
    The group ring $\bigcup_i k[X^{1/m_i},X^{-1/m_i}]$,
where $m_i$ are positive integers with $m_i|m_{i+1}$
and $k$ is a field whose characteristic does not divide $\frac{m_{i+1}}{m_i}$ for large $i$,
considered in \cite[\S 3--4]{GP-semigroup-Prufer}, is universally Japanese.
Indeed, $k[X^{1/m_i},X^{-1/m_i}]\to k[X^{1/m_{i+1}},X^{-1/m_{i+1}}]$
is \'etale for large $i$,
and these rings are finite type over $k$,
so Lemma \ref{lem:limitN2} applies.
Similarly,
$\bigcup_i \bZ[p^{1/p^i},p^{-1/p^i}]$
is universally Japanese.
\end{Exam}

\printbibliography
\end{document}